\documentclass[review]{elsarticle}

\usepackage{lineno,hyperref,amsfonts,amssymb,amsthm,amscd,amsmath,enumerate,verbatim,calc}
\usepackage{graphicx}
\newtheorem{thm}{Theorem}[section]
\newtheorem{lem}[thm]{Lemma}
\newtheorem{cor}[thm]{Corollary}
\newtheorem{prop}[thm]{Proposition}
\newtheorem{obs}[thm]{Observation}
\newtheorem{rem}[thm]{Remark}

\newtheorem{exm}[thm]{Example}

\begin{document}

\begin{frontmatter}

\title{On automorphisms and fixing number of co-normal product of graphs}
\tnotetext[mytitlenote]{This research is supported by the UPAR Grants of United Arab Emirates University, Al Ain, UAE via Grant No. G00002590 and G00003271.}

\author[mymainaddress]{Shahid ur Rehman}
\ead{shahidurrehman1982@gmail.com}
%\address{Centre for advanced studies in Pure and Applied Mathematics,
%	Bahauddin Zakariya University Multan, Pakistan.}

\author[mysecondaryaddress]{Muhammad Imran \corref{mycorrespondingauthor}}
\ead{imrandhab@gmail.com}

\author[mymainaddress]{Imran Javaid}
\cortext[mycorrespondingauthor]{Corresponding author}
\ead{imran.javaid@bzu.edu.pk}

\address[mymainaddress]{Centre for advanced studies in Pure and Applied Mathematics,
	Bahauddin Zakariya University Multan, Pakistan.}
\address[mysecondaryaddress]{Department of Mathematical Sciences, United Arab Emirates University, Al Ain, United Arab Emirates.}

\begin{abstract}
	An automorphism of a graph describes its structural symmetry and the concept of fixing number of a graph is used for breaking its symmetries (except the trivial one). In this paper, we evaluate automorphisms of the co-normal product graph $G_1\ast G_2$ of two simple graphs $G_1$ and $G_2$ and give sharp bounds on the order of its automorphism group. We study the fixing number of $G_1\ast G_2$ and prove sharp bounds on it. Moreover, we compute the fixing number of the co-normal product of some families of graphs.
\end{abstract}

\begin{keyword}
\texttt Fixing set\sep automorphism\sep co-normal product of graphs
\MSC[2010]  05C25\sep 05C76
\end{keyword}

\end{frontmatter}

\section{Introduction}
In examining a given simple graph symmetries are used to describe its structure. Each symmetry of a simple graph is known as automorphism. Mathematically, a bijective mapping $\gamma$ from the vertex set of a simple graph $H$ to itself which satisfies the property that $\gamma(a)$ and $\gamma(b)$ are neighbors if and only if $a$ and $b$ are neighbors in $H$, is known as \emph{automorphism} of $H$. The notations $V(H)$ and $E(H)$ stands for the vertex set and the edge set of a simple graph $H$. The set $Aut(H)$ comprises on all automorphisms of a graph $H$ forms a group under the operation of composition of mappings. The notation $S_{|V(H)|}$ stands for the group of all permutations of $V(H)$. Clearly, $Aut(H)$ is a subgroup of $S_{|V(H)|}$. The concept of automorphisms has its roots in graph drawing \cite{abelson}, graph coloring \cite{ramani}, model checking \cite{donald}, programming \cite{ostro}, etc.

An automorphism $\gamma \in Aut(H)$ is said to fix a vertex $\acute{a}\in V(H)$, if $\gamma(\acute{a})=\acute{a}$.
The \emph{stabilizer} of a vertex $\acute{a}\in V(H)$ is a set which contain automorphisms of $H$ under which $\acute{a}$ is fixed, denoted by $Stab(\acute{a})$. A set of automorphisms of $H$ under which each vertex of a set $\mathcal{F}\subseteq V(H)$ remains fixed is called the stabilizer of $\mathcal{F}$ (i.e., $Stab({\mathcal{F}})=\bigcap\limits_{w\in \mathcal{F}}Stab(w)$). A vertex $\acute{a}\in V(H)$ is called a \emph{fixed vertex} of $H$, if $\gamma(\acute{a}) =\acute{a}$ for all $\gamma\in Aut(H)$. For a vertex $\acute{a}\in V(H)$, the set $\{\gamma(\acute{a}): \gamma\in Aut(H)\}$ is called the \emph{orbit} of $\acute{a}$ under $Aut(H)$, denoted by $\mathcal{O}(\acute{a})$. Two vertices in the same orbit are called \emph{similar} vertices. Recently, Dehmer et al. \cite{AMC} used the cardinality of vertex orbits to find zeros of graph polynomials. For basic concepts and terminology which are not given in this paper, please see \cite{a4,l}.

Since each automorphism of a graph $H$ describes a structural symmetry of $H$ so the fixing number of $H$ is a measure of its structural complexity. Harary \cite{FH} presented several methods of fixing a graph. Erwin and Harary \cite{FD} coined the term fixing set, and for the same idea Boutin \cite{DLB} used the term determining set. This notion has its application in manipulating objects \cite{KL}. For more results on fixing number and its related parameters of graphs, see \cite{JC,DLB,MA,CRG,a3,az,FMS,ft1,lw,lw1,lw2}.
A set $\mathcal{F}\subseteq V(H)$ with the property $Stab({\mathcal{F}})=\{id_H\}$ is called a \emph{fixing set} of $H$, where $id_{H}$ is the identity automorphism of $H$. The \emph{fixing number} of $H$ \cite{FD}, denoted by $fix(H)$ defined as the cardinality of possible minimum fixing set of $H$. A graph $H$ for which $fix(H)=0$ (i.e., $Aut(H)$ is trivial) is named as \emph{rigid graph} while $H$ is a \emph{non-rigid} graph if $fix(H)> 0$ (i.e., $Aut(H)$ is non-trivial). Results on the fixing number of well known families of graphs are given in \cite{JC}. All graphs considered in this paper are non-trivial, simple, finite and non-rigid, unless otherwise stated.

%The notation $deg_H(w)$ stands for the \emph{degree} of $w\in V(H)$ which is the number of neighbors of $w$ in $H$. The set $N_H(w)$ consist of all neighbors of $w$ in $H$ named as the \emph{open neighborhood} of $w$. The set $N_H[w]= N_H(w)\cup \{w\}$ is called the \emph{closed neighborhood} of $w$. We usually use $N(w)$ and $N[w]$ when the graph $H$ is understood. Two vertices $x, w\in V(H)$ are called false (true) twins, if $N(x)= N(w)$ ($N[x]= N[w]$).
%The notations $P_n$, $C_n$ and $K_n$ stand for path, cycle and complete graph on $n$ vertices, respectively, whereas $\mathcal{K}_{m\prime,n\prime}$ denotes a complete bipartite graph on $m\prime+n\prime$ vertices. The notation $N_n$ stands for a $null$ $graph$ which consist of $n\geq 2$ vertices and no edge. The notation $S_{n+1}$ stands for a $star$ $graph$ which consist of $n+1$ vertices with $n\geq 2$ in which $n$ vertices have degree $1$ and the remaining vertex has degree $n$.

%A \emph{graph or network} $G$ is an ordered pair of two sets $V(G)$ and $E(G)$,
%called the \emph{vertex set} and the \emph{edge set} of $G$, respectively. Each element of $E(G)$ called an \emph{edge} of $G$ which represents a link between two elements of $V(G)$. Each element of $V(G)$ is called a \emph{vertex} of $G$ and two vertices $u, v\in V(G)$ are called \emph{adjacent or neighbors}, if there is an edge between them. A vertex having degree $|V(G)|-1$ is called a dominating vertex of $G$. For a subset $W$ of $V(G)$, $W^c=V(G)\setminus W$.

Until now, several products have been defined, see \cite{WI,C} for brief overview of product graphs. Ore \cite{l} presented a product graph named it \emph{Cartesian sum} of graphs that is also called the co-normal product of graphs \cite{t}. For a graph $G_1$ with vertex set $\{g^1_1,g^1_2,\ldots,g^1_{m}\}$ and a graph $G_2$ with vertex set $\{g^2_1,g^2_2,\ldots,g^2_{n}\}$, the $co$-$normal$ $product$ of $G_1$ and $G_2$ (the terminology we use) is a graph having the vertex set $V(G_1)\times V(G_2)= \{(g^1_i,g^2_j) : g^1_i\in V(G_1)$, $g^2_j\in V(G_2)\}$ and two distinct vertices $(g^1_i,g^2_j)$ and $(g^1_k,g^2_l)$ are adjacent if and only if $g^1_i$ is adjacent to $g^1_k$ in $G_1$ or $g^2_j$ is adjacent to $g^2_l$ in $G_2$, denoted by $G_1\ast G_2$. The study of product graphs is to get information about a product graph by using the knowledge of its component graphs.

A graph can be fixed by coloring a subset of its vertex set \cite{FH} as well as by using the distance between the vertices of the graph \cite{FD}.
Several properties and results are discussed in \cite{t,b,d,e,m,y} about coloring of the co-normal product of graphs. The metric dimension and strong metric dimension of $G_1\ast G_2$ were studied by Javaid et al. \cite{sss} and Kuziak et al. \cite{a0}, respectively. As a function of the order of the graph, Garijo et al. \cite{GD} analyzed the difference between its metric dimension and its fixing number. In this paper, we study the automorphism group and the fixing number of $G_1\ast G_2$.

We evaluate automorphisms of $G_1\ast G_2$ in the next section and give sharp bounds on the cardinality of its automorphism group. We also show that $G_1\ast G_2$ is a rigid graph (i.e., $Aut(G_1\ast G_2)$ is trivial) if and only if both $G_1$ and $G_2$ are indeed rigid and non-isomorphic ones. We depict conditions on $G_1$ and $G_2$ which yields that $Aut(G_1\ast G_2)= S_2$ (symmetric group). In Section 3, we study the fixing number of $G_1\ast G_2$ and prove that $max\{fix(G_1),fix(G_2)\}\leq fix(G_1\ast G_2)$ for any two graphs $G_1$ and $G_2$. We also identify conditions on $G_1$ and $G_2$ under which these bounds can be achieved. Further, we give formulae for the fixing number of the co-normal product of two path graphs, two star graphs, two complete multipartite graphs and a path with a null graph.

\section{Automorphisms of Co-normal Product}

For two graphs $G_1$ and $G_2$, $G_1\ast G_2\cong G_2\ast G_1$. In \cite{t}, Kuziak et al. characterized $G_1\ast G_2$ by using its connectivity. In \cite{a0}, the authors characterized $G_1\ast G_2$ by using its diameter. For each $(g^1_i, g^2_j)\in V(G_1\ast G_2)$ $deg(g^1_i, g^2_j)= |V(G_2)|deg(g^1_i)+(|V(G_1)|- deg(g^1_i))deg(g^2_j)$ and $N(g^1_i, g^2_j)=N(g^1_i)\times V(G_2)\cup(N(g^1_i))^{c}\times N(g^2_j)$, where $(N(g^1_i))^{c}=V(G_1)\setminus N(g^1_i)$, as given in \cite{sss}. If $deg(h)=|V(H)|-1$ for $h\in V(H)$, then $h$ is called a dominating vertex in $H$.
\begin{lem}\cite{sss}\label{r1} A vertex $(g^1_i, g^2_j)\in V(G_1\ast G_2)$ is a dominating vertex if and only if $g^1_i$ and $g^2_j$ are dominating in $G_1$ and $G_2$, respectively.
\end{lem}
We have the following observation from the definition of a rigid graph.
\begin{obs} A rigid graph has at most one
dominating vertex.
\end{obs}

In \cite{sss}, Javaid et al. proved the following result.
\begin{thm}\cite{sss}\label{twin}For any two distinct vertices $(g^1_i,g^2_j),(g^1_k, g^2_l)\in V(G_1\ast G_2)$, $N((g^1_i,g^2_j))$ $= N((g^1_k, g^2_l))$ if and only if
$N(g^1_i) = N(g^1_k)$ in $G_1$ and $N(g^2_j) = N(g^2_l)$ in $G_2$.
\end{thm}
In the next proposition, we provide restrictions on $G_1$ and $G_2$ whereby two different vertices of $G_1\ast G_2$ are true twins.
\begin{prop}\label{g4}Two distinct vertices $(g^1_i, g^2_j), (g^1_k, g^2_l)$ of $G_1\ast G_2$ are true twins, if and only if one of the following holds:\\
(1) For $g^1_i\neq g^2_k$ and $g^2_j\neq g^2_l$, the vertices $g^1_i$, $g^1_k$ and $g^2_j$, $g^2_l$ are dominating vertices of $G_1$ and $G_2$, respectively. \\
(2) For $g^1_i=g^1_k$ and $g^2_j\neq g^2_l$, vertex $g^1_i$ is a dominating vertex and $N[g^2_j]=N[g^2_l]$.\\
(3) For $g^1_i\neq g^1_k$ and $g^2_j= g^2_l$, vertex $g^2_j$ is a dominating vertex and $N[g^1_i]=N[g^1_k]$.
\end{prop}
\begin{proof}
(1) Suppose $(g^1_i, g^2_j), (g^1_k, g^2_l)$ are true
twins in $G_1\ast G_2$ and $g^1_i\neq g^1_k$, $g^2_j\neq g^2_l$. In order to prove that $g^1_i, g^1_k$ are dominating in $G_1$ and $g^2_j, g^2_l$ are dominating in $G_2$. Suppose on contrary that at least one of these vertices is not a dominating vertex. In particular, suppose $g^1_i$ is the only vertex which is not a dominating vertex in $G_1$, then Lemma \ref{r1} gives that $(g^1_i, g^2_j)$ is not a dominating vertex, a contradiction.

Converse arguments follows from Lemma \ref{r1}.\\
(2) Suppose $(g^1_i, g^2_j), (g^1_k, g^2_l)$ are true
twins in $G_1\ast G_2$ and $g^1_i=g^1_k$, $g^2_j\neq g^2_l$.
In order to prove that $g^1_i$ is a dominating vertex in $G_1$ and $g^2_j, g^2_l$ are true twins in $G_2$ suppose on contrary that $g^1_i$ is not a dominating vertex in  $G_1$ and $N[g^2_j]=N[g^2_l]$, then there exist at least one vertex say $g^1_r\in
V(G_1)\setminus\{g^1_i\}$ such that $g^1_r\notin N(g^1_i)$ and $(g^1_r, g^2_j)\notin N(g^1_i, g^2_j)$ but $(g^1_r, g^2_j)\in N(g^1_i,g^2_l)$, hence $N[(g^1_i, g^2_j)]\neq N[(g^1_i, g^2_l)]$, a contradiction. Now suppose that $g^1_i$ is a dominating vertex of $G_1$ and $N[g^2_j]\neq N[g^2_l]$, then by the definition of the co-normal product $N[(g^1_i, g^2_j)]\neq N[(g^1_i, g^2_l)]$, a contradiction. Hence, $g^1_i$ is a dominating vertex and $g^2_j, g^2_l$ are true twins in $G_2$.

Converse arguments follows from the definition of the co-normal product of graphs.\\
(3) The proof directly follows from the commutative property of the co-normal product of two graphs.
\end{proof}

Automorphisms of a graph provide basic knowledge about its symmetrical structure. For every $\gamma\in Aut(G_1)$ and $\eta\in Aut(G_2)$, mappings $\lambda_{G_1}$ and $\lambda_{G_2}$ on $G_1\ast G_2$ defined as $\lambda_{G_1}(g^1_i, g^2_j)=(\gamma(g^1_i), g^2_j)$ and $\lambda_{G_2}(g^1_i, g^2_j)=(g^1_i, \eta(g^2_j))$ are automorphisms of $G_1\ast G_2$. The next theorem gives automorphisms of $G_1\ast G_2$.

\begin{thm}\label{g1}Let $G_1$ and $G_2$ be two graphs and
$\lambda : V(G_1\ast G_2)\rightarrow V(G_1\ast G_2)$ be a mapping, then the following assertions holds:\\
(1) For every automorphism $\alpha$ of $G_1$ and $\beta$ of $G_2$, the mapping $\lambda$ defined as $\lambda(g^1_i, g^2_j)=(\alpha(g^1_i), \beta(g^2_j))$ for all $(g^1_i, g^2_j)\in V(G_1\ast G_2)$, is an automorphism of $G_1\ast G_2$.\\
(2) If $G_1\cong G_2$, then for every isomorphism $\alpha$ of $G_1$ to $G_2$ and $\beta$ of $G_2$ to $G_1$, the mapping $\lambda$
defined as $\lambda(g^1_i, g^2_j)=(\beta(g^2_j),\alpha(g^1_i))$ for all $(g^1_i, g^2_j)\in V(G_1\ast G_2)$,
is an automorphism of $G_1\ast G_2$.
\end{thm}
\begin{proof}
(1) Consider two distinct vertices $(g^1_i,g^2_j), (g^1_k, g^2_l)\in V(G_1\ast G_2)$ and suppose $(g^1_i,g^2_j)$ is adjacent to $(g^1_k, g^2_l)$, then $g^1_i$ is adjacent to $g^1_k$ in $G_1$ or $g^2_j$ is adjacent to $g^2_l$ in $G_2$. Since $\alpha$ and $\beta$ are automorphisms of $G_1$ and $G_2$, respectively, so $\alpha(g^1_i)$ is adjacent to $\alpha(g^1_k)$ or $\beta(g^2_j)$ is adjacent to $\beta(g^2_j)$. Thus $\lambda(g^1_i,g^2_j)$ is adjacent to $\lambda (g^1_k, g^2_l)$. Now, suppose $(g^1_i,g^2_j)$ is not adjacent to $(g^1_k, g^2_l)$, then $g^1_i$ is not adjacent to $g^1_k$ in $G_1$ and $g^2_j$ is not adjacent to $g^2_l$ in $G_2$. Clearly, $\lambda(g^1_i,g^2_j)$ is not adjacent to $\lambda (g^1_k, g^2_l)$, which gives that $\lambda(g^1_i, g^2_j)=(\alpha(g^1_i), \beta(g^2_j))$, is a homomorphism of $G_1\ast G_2$. Also, $\lambda$ is a bijective mapping because $\alpha$ and $\beta$ are bijective mappings. Hence, $\lambda(g^1_i, g^2_j)=(\alpha(g^1_i), \beta(g^2_j))$ is an automorphism of $G_1$.

(2) The proof follows from the fact that every isomorphism
is an automorphism when the labels are removed from $G_1$ and $G_2$.
\end{proof}
For an automorphism $\alpha\in Stab(g^1_i)$ and $\beta\in Stab(g^2_j)$, Theorem \ref{g1} (1) gives that mapping $\lambda(g^1_i, g^2_j)=(\alpha(g^1_i), \beta(g^2_j))$ is an automorphism of $G_1\ast G_2$ and $\lambda\in Stab(g^1_i, g^2_j)$. Hence, we have the following corollary:

\begin{cor}\label{g3}For any two graphs $G_1$ and $G_2$, $Stab(g^1_i)\times
Stab(g^2_j)\subseteq Stab(g^1_i, g^2_j)$ for each $(g^1_i, g^2_j)\in V(G_1\ast
G_2)$.
\end{cor}

For a positive integer $k\geq 2$, we define a product graph obtained
from $G_1, G_2,\ldots,$ $G_k$ with $|V(G_i)|=m_i\geq 2$ for each $i$.
The graph having the vertex set $V(G_1)\times V(G_2)\times \ldots
\times V(G_k)$, where $(g_1, g_2, \ldots, g_k)$ is adjacent to
$(\acute{g_1},\acute{g_2}, \ldots, \acute{g_k})$ whenever $g_i$ is
adjacent to $\acute{g_i}$ in $G_i$ for some $1\leq i\leq k$, is
called the \emph{generalized co-normal product graph} of $G_1,
G_2,\ldots,$ $G_k$, denoted by $G_1\ast G_2\ast \ldots\ast G_k$. Let
$\mathcal{G}=G_1\ast G_2\ast\ldots\ast G_k$, where
$|V(G_i)|\geq 2$ for each $i$. For any
vertex $x\in V(G_i)$ for some $1\leq i\leq k$, we define a vertex
set $\mathcal{G}(x)=\{(g^1,
g^2,\ldots,g^{i-1},x,g^{i+1},\ldots,g^k)| g^j\in V(G_j), j\neq i\}$.
Note that, $\langle\mathcal{G}(x)\rangle\cong G_1\ast
G_2\ast\ldots\ast G_{i-1}\ast G_{i+1}\ast\ldots\ast G_k$ for each $x\in V(G_i)$. In particular for $k=2$, $\mathcal{G}=G_1\ast G_2$ is the usual co-normal product graph and for $k>2$, $\mathcal{G}$ is a natural generalization of
usual co-normal product, therefore its properties too. Theorem \ref{g1} can be generalized to the generalized co-normal product of graphs as:
\begin{thm}\label{g2} Let $\pi\in S_k$ be a permutation such that
$\psi_i: V(G_i)\rightarrow V(G_{\pi(i)})$ is an isomorphism for each
$1\leq i\leq k$. The mapping $\lambda: V(\mathcal{G}) \rightarrow
V(\mathcal{G})$, defined as $\lambda(g^1,g^2,\ldots,g^k)=
(\psi_{\pi^{-1}(1)}(g^{\pi^{-1}(1)}),
\psi_{\pi^{-1}(2)}(g^{\pi^{-1}(2)}),\ldots,
\psi_{\pi^{-1}(k)}(g^{\pi^{-1}(k)})$ is an automorphism of
$\mathcal{G}$.
\end{thm}
\begin{cor}\label{n2}Let $\mathcal{G}$ be the co-normal product
graph of $k\geq 2$ graphs. For every vertex $(g^1, g^2,\ldots,
g^k)\in V(\mathcal{G})$, we have $Stab(g^1)\times Stab(g^2)\times
\ldots\times Stab(g^k)\subseteq Stab(g^1, g^2,\ldots, g^k)$.
\end{cor}

For the co-normal product graph $\mathcal{G}=G_1\ast G_2$, we define the vertex sets $\mathcal{G}(g^2_j)=\{(g^1_i,g^2_j)|g^1_i\in V(G_1)\}$ and $\mathcal{G}(g^1_i)=\{(g^1_i,g^2_j)|g^2_j\in V(G_2)\}$ for each $g^2_j\in V(G_2)$ and $g^1_i\in V(G_1)$. Moreover, $\langle\mathcal{G}(g^2_j)\rangle\cong G_1$ and $\langle\mathcal{G}(g^1_i)\rangle\cong G_2$ for each $g^2_j\in V(G_2)$ and $g^1_i\in V(G_1)$. Remark \ref{rem}, immediately follows from the fact that every automorphism is an isometry \cite{big}, i.e., for every $\alpha\in Aut(G_1)$ and $x,y\in V(G_1)$, $d(x,y)=d(\alpha(x), \alpha(y))$.
\begin{rem}\label{rem}For every automorphism $\alpha\in Aut(G)$ and every subgraph $H$ of $G$, $\langle\alpha(V(H))\rangle\cong H$.
\end{rem}
Let $id_{G_1}$ and $id_{G_2}$ denotes the identity automorphism of $G_1$ and $G_2$, respectively. The automorphism $\lambda(g^1_i,g^2_j)=(\alpha(g^1_i), \beta(g^2_j))$ of $G_1\ast G_2$ as defined in Theorem \ref{g1}(1) is called a $rotation$ about $G_1$ (or $G_2$) if $\alpha\neq id_{G_1}$ and $\beta=id_{G_2}$ or $\alpha= id_{G_1}$ and $\beta\neq id_{G_2}$. For $G_1\cong G_2$, automorphism $\lambda(g^1_i,g^2_j)=(\beta(g^2_j), \alpha(g^1_i))$ as defined in Theorem \ref{g1}(2) exists and is called a $flip$ of $G_1\ast G_2$. For two distinct vertices $u,v\in V(H)$, an automorphism $\alpha\in Aut(H)$ such that $\alpha(u)=v$, $\alpha(v)=u$ and $\alpha(x)=x$ for all $x\in V(H)\setminus \{u,v\}$ is called a $(u, v)-interchange$. All results given in this paper for $G_1\ast G_2$ are also hold for $G_2\ast G_1$ due to the commutative property of co-normal product
and the results can be generalized to the generalized co-normal
product graph of $k> 2$ graphs. For a fixed vertex $g^1_k\in V(G_1)$ and any $\psi\in Aut(G_2)$, we define a mapping $\lambda$ on $V(G_1\ast G_2)$ as:
\begin{equation}
\lambda(g^1_i,g^2_j)\in V(G_1\ast G_2)\rightarrow \left\{
            \begin{array}{ll}
          (g^1_k,\psi(g^2_j))&  \,\,\,\,\,\,\,\,\,\,\,\,\  if\,\,\,\ i=k,\\
           (g^1_i, g^2_j)&  \,\,\,\,\,\,\,\,\,\,\,\,\  otherwise.
            \end{array}
             \right.
\end{equation}
Note that if $\psi=id_{G_2}$, then $\lambda$ is a trivial automorphism of $G_1\ast G_2$. Also, mapping $\lambda$ as defined in (1) is a non-trivial automorphism for $P_3\ast P_4$ as shown in Figure 1 for $k=2$ and $id_{G_2}\neq \psi\in Aut(P_4)$.
\begin{figure}[h]
        \centerline
        {\includegraphics[width=08cm]{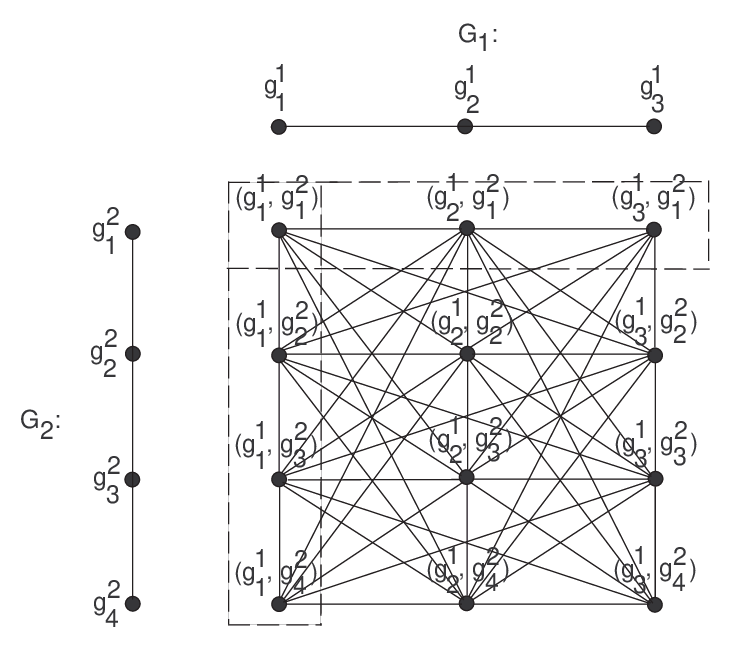}}
        \label{p}
        \caption{The co-normal product graph of $G_1=P_3$ and $G_2=P_4$.}\label{p}
\end{figure}
The proof of next proposition is straightforward.
\begin{prop}\label{g5}For each dominating vertex $g^1_k$ of $G_1$ and for every non-trivial automorphism $\psi$ of $G_2$, mapping $\lambda$ as defined in (1) built from $\psi$ is a non-trivial automorphism of $G_1\ast G_2$.
\end{prop}

Note that $\lambda$ as defined in (1) is not an automorphism for graph $P_4\ast P_4$ for any $k$ as shown in Figure 2. In next theorem, we give conditions on $G_1$ and $G_2$ so that mapping $\lambda$ as defined in (1) is an automorphism of $G_1\ast G_2$ when $g^1_k$ is not a dominating vertex.
\begin{thm}\label{g6}Let $g^1_k\in V(G_1)$ is not a dominating vertex  and $\psi\in Aut(G_2)$ is a non-trivial automorphism  of $G_2$, then $\lambda$ as defined in (1) built from $\psi$ is a non-trivial automorphism of $G_1\ast G_2$ if and only if $N(g^2_j)=N(\psi(g^2_j))$ in $G_2$ for each $g^2_j\in V(G_2)$.
\end{thm}
\begin{proof}Let $g^1_k\in V(G_1)$ and $\psi\in Aut(G_2)$ is a non-trivial automorphism and $\lambda$ as defined in (1) built from $\psi$ is a non-trivial automorphism
of $G_1\ast G_2$. Since $\psi$ is non-trivial, therefore there exist $g^2_j\neq g^2_l\in V(G_2)$ such that $\psi(g^2_j)=g^2_l$. Suppose $N(g^2_j)\neq N(\psi(g^2_j))$ in $G_2$, then there exist at least one vertex $g^2_s\in
V(G_2)$ such that $g^2_s \in N(g^2_j)$ and $g^2_s \notin
N(\psi(g^2_j))$ or $g^2_s \notin N(g^2_j)$ and $g^2_s \in
N(\psi(g^2_j))$. Also, $g^1_k\in V(G_1)$ is not a dominating vertex so there exist at least one vertex $g^1_i$ such that $g^1_i\notin N(g^1_k)$. The definition of $\lambda$ gives that $\lambda((g^1_i,g^2_s))=(g^1_i,g^2_s)$ and $\lambda((g^1_k,g^2_j))= (g^1_k,g^2_l)$, moreover, $d(\lambda((g^1_i,g^2_l)),\lambda((g^1_k,g^2_j)))> 1$ and $d((g^1_i,g^2_s), (g^1_k,g^2_j))=1$, which gives that $\lambda$ is not an isometry, a contradiction. Hence, $N(g^2_j)=N(\psi(g^2_j))$ for each $g^2_j\in V(G_2)$.

Conversely, suppose $N(g^2_j)= N(\psi(g^2_j))$ for each $g^2_j\in V(G_2)$. Theorem \ref{twin} gives that for each $g^1_i\in
V(G_1)$, vertices $(g^1_i,g^2_j)$ and $(g^1_i,\psi(g^2_j))$ are
also false twins in $G_1\ast G_2$. Hence, $\lambda$ is a non-trivial automorphism of $G_1\ast G_2$.
\end{proof}
Using the structure of co-normal product of two graphs and Theorem
\ref{g6}, we have the following corollary.
\begin{cor}\label{g7}
For any two graphs $G_1$ and $G_2$, $Aut(G_1\ast G_2)$ has a
$((g^1_i, g^2_j), (g^1_k, g^2_l))-interchange$ if and only if one of the following holds:\\
(1) $g^1_i=g^1_k$, $g^2_j\neq g^2_l$ and $N(g^2_j)=N(g^2_l)$.\\
(2) $g^1_i\neq g^1_k$, $g^2_j= g^2_l$ and $N(g^1_i)=N(g^1_k)$.\\
(3) $g^1_i=g^1_k$, $g^2_j\neq
g^2_l$ with $g^1_i$ dominating in $G_1$ and
$N[g^2_j]=N[g^2_l]$.\\
(4) $g^1_i\neq g^1_k$, $g^2_j= g^2_l$ with $g^2_j$ dominating in
$G_2$ and
$N[g^1_i]=N[g^1_k]$.\\
(5) $g^1_i\neq g^1_k$, $g^2_j\neq g^2_l$ with $N(g^1_i)=N(g^1_k)$ and
$N(g^2_j)=N(g^2_l)$.\\
(6) $g^1_i\neq g^1_k$, $g^2_j\neq g^2_l$ with $g^1_i,g^1_k$ are dominating in $G_1$ and $g^2_j, g^2_l$
are dominating in $G_2$.
\end{cor}

The definition of co-normal product of graphs gives that $Aut(G_1\ast G_2)$ is isomorphic to $S_{|V(G_1\ast G_2)|}$ (symmetric group) if and only if both $G_1$ and $G_2$ are complete or null graphs. Moreover, Theorem \ref{g1} (1) gives that $Aut(G_1)\times Aut(G_2)\subseteq
Aut(G_1\ast G_2)$ for any two graphs $G_1$ and $G_2$. Hence, we have the following bounds:
\[|Aut(G_1)||Aut(G_2)|\leq |Aut(G_1\ast G_2)|\leq (|V(G_1)||V(G_2)|)!\]

In the next theorem, we give conditions on $G_1$ and $G_2$ so that $G_1\ast G_2$ is a rigid graph and the lower bound is attained.

\begin{thm}\label{f1} For any two graphs $G_1$ and $G_2$, graph $G_1\ast G_2$ is a rigid graph if and only if $G_1$ and $G_2$ are non-isomorphic rigid graphs.
\end{thm}
\begin{proof} Suppose $G_1\ast G_2$ is a rigid graph. In order to prove that $G_1$ and $G_2$ are both rigid graphs, suppose on contrary that $G_1$ is non-rigid and $\alpha\in Aut(G_1)$ is a non-trivial automorphism, then $\bar{\lambda}$ built from $\alpha$ and $\beta=id_{G_2}$ described in Theorem \ref{g1} (1) is a non-trivial automorphism of $G_1\ast G_2$, a contradiction. A similar argument holds, when $G_2$ is a non-rigid graph. Hence, $G_1$ and $G_2$ must be rigid graphs. Now, suppose $G_1\cong G_2$ and $\alpha$ is an isomorphism from $G_1$ to $G_2$ and $\beta$ is an isomorphism from $G_2$ to $G_1$, then there exists a non-trivial automorphism $\bar{\lambda}\in Aut(G_1\ast G_2)$ built from $\alpha$ and $\beta$ as described in Theorem \ref{g1} (2), a contradiction. Hence $G_1$ and $G_2$ are non-isomorphic graphs.

Conversely, suppose that $G_1$ and $G_2$ are non-isomorphic rigid graphs and $\lambda\in Aut(G_1\ast G_2)$ such that $\lambda(g^1, g^2)=(g^1_i, g^2_j)$ for some $(g^1, g^2)\in V(G_1\ast G_2)$. Since $G_1\ncong G_2$ so $\lambda$ is not a flip and $\lambda$ is not an interchange because $G_1$ and $G_2$ has no dominating vertices. Now to prove that $\lambda$ is trivial, we study the following cases:\\
Case 1) Suppose $g^1=g^1_i$ and $g^2\neq g^2_j$. Since $\lambda$ is not a flip and $\langle\lambda(\mathcal{G}(g^1))\rangle\cong G_2$ so there exists a non-trivial automorphism $\beta$ of $G_2$ such that $\beta(g^2)=g^2_j$ which contradicts that $G_2$ is a rigid graph.\\
Case 2) Suppose $g^1\neq g^1_i$ and $g^2=g^2_j$. Since $\lambda$ is not a flip and $\langle\lambda(\mathcal{G}(g^1))\rangle\cong G_2$ so there exists a non-trivial automorphism $\alpha$ of $G_1$ such that $\alpha(g^1)=g^1_i$ which contradicts that $G_1$ is a rigid graph.\\
Case 3) Suppose $g^1\neq g^1_i$ and $g^2\neq g^2_j$. Since $\lambda$ is not a flip and $\langle\lambda(\mathcal{G}(g^2))\rangle\cong G_1$ and $\langle\lambda(\mathcal{G}(g^1))\rangle\cong G_2$ so there exists a non-trivial automorphism $\alpha$ of $G_1$ such that $\alpha(g^1)=g^1_i$ and a non-trivial automorphism $\beta$ of $G_2$ such that $\beta(g^2)=g^2_j$ which contradicts that $G_1$ and $G_2$ are rigid graphs.\\

Concluding all above cases, we have $\lambda(g^1,g^2)=(g^1,g^2)$ for all $(g^1,g^2)\in V(G_1\ast G_2)$ i.e., $\lambda$ is trivial which completes the proof.
\end{proof}

In the next two theorems, we give conditions on $G_1$ and $G_2$ so that $Aut(G_1\ast G_2)=Aut(G_1)\times Aut(G_2)$ when at least one of $G_1$ and $G_2$ is a non-rigid graph.
\begin{thm}\label{nm1}For a rigid graph $G_1$ and a non-rigid graph $G_2$,
$Aut(G_1\ast G_2)=Aut(G_1)\times Aut(G_2)$ if and only if $G_1$ has no
dominating vertex and $G_2$ has no false twins.
\end{thm}
\begin{proof}Suppose $Aut(G_1\ast G_2)= Aut(G_1)\times Aut(G_2)$. In order to prove that $G_1$ has no dominating vertex suppose on contrary that $G_1$ has a dominating vertex $g^1_k\in V(G_1)$ and $\bar{\psi}\in Aut(G_2)$ is non-trivial, then mapping $\bar{\lambda}\in Aut(G_1\ast G_2)$ built from $\bar{\psi}$ as described in Proposition \ref{g5} is a non-trivial automorphism of $G_1\ast G_2$ such that $\bar{\lambda}\notin Aut(G_1)\times Aut(G_2)$, a contradiction. Hence, $G_1$ has no dominating vertex. Now, suppose $g^2_j\neq g^2_l\in V(G_2)$ are false twins in $G_2$, then Corollary \ref{g7} (1) gives that for each $g^1_i\in V(G_1)$ there
exists a $((g^1_i, g^2_j), (g^1_i, g^2_l))-interchange$ in $Aut(G_1\ast G_2)$ which does not belong to $Aut(G_1)\times Aut(G_2)$, a contradiction. Hence, $G_2$ has no false twins.

Conversely, suppose $G_1$ is a rigid graph without dominating vertex and $G_2$ is without false twins. Theorem \ref{g1} (1) gives that $Aut(G_1)\times Aut(G_2)\subseteq Aut(G_1\ast G_2)$. Consider  $\lambda\in Aut(G_1\ast G_2)$ since $G_1$ and $G_2$ are non-isomorphic graphs so $\lambda$ is not a flip and Corollary \ref{g7} gives that $\lambda$ is not an interchange because $G_1$ and $G_2$ has no false twins. Further, $G_1$ has no dominating vertex so for every non-trivial automorphism $\psi\in Aut(G_2)$, $\lambda$ does not have the form described in Proposition \ref{g5} built from $\psi$. Suppose $\lambda(g^1_i, g^2_j)=(g^1_k, g^2_l)$ such that $g^1_i\neq g^1_k$, then there exists a non-trivial automorphism $\alpha$ of $G_1$ for which $\alpha(g^1_i)=g^1_k$, a contradiction to the fact that $G_1$ is a rigid graph. Hence $\lambda(g^1_i, g^2_j)=(g^1_i, g^2_l)$ i.e., $\lambda$ is constant on the first coordinate. Now suppose $\lambda\notin Aut(G_1)\times Aut(G_2)$ and $g^2_j$ is similar to $g^2_l$ in $G_2$, then there exists a vertex $(g^1_i, g^2_r)\in \mathcal{G}(g^1_i)$ such that $\lambda(g^1_i, g^2_r)=(g^1_i, g^2_t)$ and $G_2$ has no automorphism $\beta$ for which $\beta(g^2_j)=g^2_l$ and $\beta(g^2_r)=g^2_t$ which yields that $\langle\lambda(\mathcal{G}(g^1_i))\rangle\ncong G_2$ but $\langle\mathcal{G}(g^1_i)\rangle\cong G_2$, a contradiction to the fact that $\lambda$ is an automorphism by Remark \ref{rem}. Now suppose that $g^2_j$ is not similar to $g^2_l$ in $G_2$, then $\langle\lambda(\mathcal{G}(g^1_i))\rangle\ncong G_2$ because $\lambda$ is a bijective mapping, a contradiction to the fact that $\lambda$ is an automorphism by Remark \ref{rem}.
\end{proof}

\begin{thm}\label{g8}For any two non-rigid, non-isomorphic graphs $G_1$ and $G_2$, $Aut(G_1\ast G_2)$ $= Aut(G_1)\times Aut(G_2)$ if and only if neither $G_1$ nor $G_2$ has dominating vertex and false twins.
\end{thm}
\begin{proof}Let $Aut(G_1\ast G_2)= Aut(G_1)\times Aut(G_2)$, we will prove that neither $G_1$ nor $G_2$ has dominating vertex or false twins. First, suppose that $G_1$ has a
dominating vertex say $g^1_k$, then for every non-trivial automorphism $\bar{\psi}\in Aut(G_2)$, mapping $\bar{\lambda}$ built from $\bar{\psi}$ described in Proposition \ref{g5} is a non-trivial automorphism of $G_1\ast G_2$ such that $\lambda\notin Aut(G_1)\times Aut(G_2)$. Now, suppose $G_1$ or $G_2$ has false twins. In particular, suppose $g^2_j\neq g^2_l\in V(G_2)$ are false twins, then Corollary \ref{g7} (1) gives that for each $g^1_i\in V(G_1)$ there
exists a $((g^1_i, g^2_j), (g^1_i, g^2_l))-interchange$ in $Aut(G_1\ast G_2)$ which does not belong to $Aut(G_1)\times Aut(G_2)$, a contradiction. Hence, neither $G_1$ nor $G_2$ has false twins and dominating vertices.

Conversely, suppose $G_1$ and $G_2$ are non-isomorphic, non-rigid graphs without false twins and dominating vertices. Theorem \ref{g1} (1) gives that $Aut(G_1)\times Aut(G_2)\subseteq Aut(G_1\ast G_2)$. Now to prove that $Aut(G_1\ast G_2)\subseteq Aut(G_1)\times Aut(G_2)$ assume contrary that there exists $\lambda\in Aut(G_1\ast G_2)$ such that $\lambda\notin Aut(G_1)\times Aut(G_2)$. Now $G_1$ and $G_2$ are non-isomorphic graphs so $\lambda$ is not a flip and Corollary \ref{g7} gives that $\lambda$ is not an interchange because $G_1\ast G_2$ has no false twins. Further, for each non-trivial automorphism $\psi\in Aut(G_2)$, $\lambda$ does not have the form described in Proposition \ref{g5} built from $\psi$. A similar argument holds for each non-trivial automorphism $\phi\in Aut(G_1)$ due to the commutative property of co-normal product of $G_1$ and $G_2$. Suppose $\lambda(g^1_i, g^2_j)=(g^1_k, g^2_l)$ then we study two cases as follows:\\
Case 1) Suppose $g^1_i$ is similar to $g^1_k$ in $G_1$ and $g^2_j$ is similar to $g^2_l$ in $G_2$. We have two possibilities:\\
(1) Since $\lambda\notin Aut(G_1)\times Aut(G_2)$ so there exists a vertex $(g^1_i, g^2_r)\in \mathcal{G}(g^1_i)$ such that $\lambda(g^1_i, g^2_r)=(g^1_k, g^2_t)$ and $G_2$ has no automorphism $\beta$ for which $\beta(g^2_j)=g^2_l$ and $\beta(g^2_r)=g^2_t$ which yields that $\langle\lambda(\mathcal{G}(g^1_i))\rangle\ncong G_2$ but $\langle\mathcal{G}(g^1_i)\rangle\cong G_2$, a contradiction to the fact that $\lambda$ is an automorphism by Remark \ref{rem}.\\
(2) Since $\lambda\notin Aut(G_1)\times Aut(G_2)$ so there exists a vertex $(g^1_r, g^2_j)\in \mathcal{G}(g^2_j)$ such that $\lambda(g^1_r, g^2_j)=(g^1_t, g^2_l)$ and $G_1$ has no automorphism $\alpha$ for which $\alpha(g^1_i)=g^1_k$ and $\alpha(g^1_r)=g^1_t$ which yields that $\langle\lambda(\mathcal{G}(g^2_j))\rangle\ncong G_1$ but $\langle\mathcal{G}(g^2_j)\rangle\cong G_1$, a contradiction to the fact that $\lambda$ is an automorphism by Remark \ref{rem}.\\
Case 2) Suppose $\lambda(g^1_i, g^2_j)=(g^1_k,g^2_l)$ such that $g^1_i$ is not similar to $g^1_k$ in $G_1$ then $\langle\lambda(\mathcal{G}(g^1_i))\rangle\ncong G_2$ because $\lambda$ is a bijective mapping, a contradiction to the fact that $\lambda$ is an automorphism by Remark \ref{rem}. Now suppose $g^2_j$ is not similar to $g^2_l$ in $G_2$ then $\langle\lambda(\mathcal{G}(g^2_j))\rangle\ncong G_1$ because $\lambda$ is a bijective mapping, a contradiction to the fact that $\lambda$ is an automorphism by Remark 2.10.
\end{proof}

A path graph with three vertices has false twins and each path graph with at least four vertices has no false twins. Hence, we have the following corollary from Theorem \ref{nm1}.
\begin{cor}For a rigid graph $G_1$ with no dominating vertex and $G_2\cong P_n$;
$n\geq4$, $Aut(G_1\ast G_2)\cong S_2$.
\end{cor}

The next corollary is a direct consequence of Theorem \ref{g8}.
\begin{cor}\label{n1}Let $G_1$ and $G_2$ be two non-isomorphic,
non-rigid graphs. For each $(g^1_i, g^2_j)\in V(G_1\ast G_2)$,
$Stab(g^1_i, g^2_j)$ $= Stab(g^1_i)\times Stab(g^2_j)$ if and only
if neither $G_1$ nor $G_2$ has dominating vertex and false twins.
\end{cor}

In the following theorem, we give the automorphism group of $G_1\ast G_2$, when $G_1$ and $G_2$ are isomorphic rigid graphs.
\begin{thm}\label{s1}For any two rigid isomorphic graphs $G_1$ and $G_2$ with no dominating vertices,
$Aut(G_1\ast G_2)\cong S_2$.
\end{thm}
\begin{proof}Since $G_1\cong G_2$ so there exists isomorphisms $\phi$ and $\psi$ from $G_1$ to $G_2$ and $G_2$ to $G_1$, respectively, such that mapping $\lambda$ defined in Theorem \ref{g1} (2) built from $\phi$ and $\psi$ is a non-trivial automorphism of $G_1\ast G_2$. Given that $G_1$ and $G_2$ are rigid graphs which gives that $\phi$ and $\psi$ are unique isomorphisms. Hence, $G_1\ast G_2$ has a unique non-trivial automorphism.
\end{proof}

\section{Fixing Number of Co-normal Product}
This section is devoted towards the study of the fixing number of the co-normal product of two graphs.
For a rigid graph $H$, $fix(H)=0$ and Theorem \ref{f1} gives that $fix(G_1\ast G_2)=0$ if and only if
$G_1$ and $G_2$ are non-isomorphic rigid graphs. Also, Theorem
\ref{s1} gives that for two isomorphic rigid graphs $G_1$ and $G_2$,
$fix(G_1\ast G_2)=1$. The graph $G_1+G_2$ obtained from two graphs $G_1$ and $G_2$, has $V(G_1)\cup V(G_2)$ as its vertex set and $E(G_1)\cup E(G_2)\cup \{g^1_i\sim g^2_j| g^1_i\in V(G_1), g^2_j\in V(G_2)\}$ as its edge set, is named as join graph. The next observation and Proposition \ref{1} will help in proving our next results.
\begin{obs} If $G_1$ has a dominating vertex $g^1_i$,
then $G_1\ast G_2= \langle V(G_1)\setminus \{g^1_i\}\rangle\ast G_2 +
\langle\mathcal{G}(g^1_i)\rangle$.
\end{obs}
\begin{prop}\label{1}For a rigid graph $G$ with a dominating vertex $v$, the induced subgraph $G^\prime=\langle V(G)\setminus \{v\}\rangle$ of $G$, is also a rigid graph.
\end{prop}
\begin{proof}Let $\psi$ be an
automorphism of $G^\prime$. The mapping $\psi_v: V(G)\rightarrow
V(G)$, defined as $\psi_v(v)=v$ and $\psi_v(u)=\psi(u)$ for $u\in
V(G)\setminus \{v\}$, is a non-trivial automorphism of $G$ if and only if $\psi$ is a non-trivial automorphism of $G^\prime$.
Hence, $G^\prime$ is rigid.
\end{proof}
For a graph $H$, a graph $H^c$ with $V(H)$ as its vertex set and two vertices are adjacent whenever they are not adjacent in $H$ is named as  \emph{complement graph} of $H$. The \emph{strong
product} of two graphs $G_1$ and $G_2$ is a graph with the vertex set
$V(G_1)\times V(G_2)$ and the adjacency is defined as: $(g^1_i,
g^2_j)\sim (g^1_k, g^2_l)$ if and only if $g^1_i=g^1_k$ and
$g^2_j\sim g^2_l$, $g^1_i\sim g^1_k$ and $g^2_j= g^2_l$ or
$g^1_i\sim g^1_k$ and $g^2_j\sim g^2_l$, denoted as $G_1\boxtimes
G_2$. In \cite{a0}, Kuziak et al. proved that $(G_1\ast G_2)^c=
G_1^c\boxtimes G_2^c$, where $G_1^c\boxtimes G_2^c$ is the strong
product of $G_1^c$ and $G_2^c$. Moreover, $fix(G_1)=fix(G_1^c)$ \cite{JC}. This yields the following.
\begin{thm} For any two graphs $G_1$ and $G_2$, $fix(G_1\ast G_2)=fix(G_1^c\boxtimes G_2^c)$.
\end{thm}
The following observation follows from the definition of co-normal product of graphs.
\begin{obs}\label{o2}
	If $G_1$ has $r$ components
	$G^1_1,G^1_2,\ldots,G^1_r$ and $G_2$ has $s$ components $G^2_1,
	G^2_2,$$\ldots,$ $G^2_s$, then for each $x\in V(G^1_i)\times V(G^2_j)$
	and $y,z\in V(G^1_k)$ $\times$ $ V(G^2_l)$, we have $d(x,y)=d(x,z)$ for all
	$i\neq k$ and $j\neq l$ in $G_1\ast G_2$.
\end{obs}

The next lemma follows from Observation \ref{o2}.
\begin{lem}\label{2}If $g^1_i$ is a dominating vertex in $G_1$ then vertices of the
	class $\mathcal{G}(g^1_i)$ are not fixed by
	any vertex from $V(G_1\ast G_2)\setminus \mathcal{G}(g^1_i)$.
\end{lem}

In the next theorem, we give sharp bounds on the fixing number of the co-normal product of two graphs.
\begin{thm}\label{fix1}For any two graphs $G_1$ and $G_2$,
\[max\{fix(G_1),fix(G_2)\}\leq fix(G_1\ast G_2)\leq |V(G_1)||V(G_2)|-1.\]
\end{thm}
\begin{proof}Let $\mathcal{F}\subseteq V(G_1\ast G_2)$, be an arbitrary set such that $|\mathcal{F}|< max\{fix(G_1)$, $ fix(G_2)\}$. Without loss of generality, suppose that $max\{fix(G_1), fix(G_2)\}= fix(G_2)$ and define $\mathcal{F}_1=\{g^1_i\in V(G_1)| (g^1_i, g^2_j)\in \mathcal{F}$ for some $g^2_j\in V(G_2)\}\subseteq V(G_1)$ and $\mathcal{F}_2=\{g^2_j\in V(G_2)| (g^1_i, g^2_j)\in \mathcal{F}$ for some $g^1_i\in V(G_1)\}\subseteq V(G_2)$.
Since $|\mathcal{F}|<fix(G_2)$ so $\cap_{g^2_j\in \mathcal{F}_2}$ $Stab(g^2_j)\neq \{id_{G_2}\}$. Let $id_{G_2}\neq \psi\in \cap_{g^2_j\in \mathcal{F}_2}$ $Stab(g^2_j)$, then mapping $\lambda=(id_{G_1}, \psi)$ is a non-trivial automorphism of $G_1\ast G_2$. Hence, $\mathcal{F}$ is not a fixing set for $G_1\ast G_2$ and $max\{fix(G_1),fix(G_2)\}\leq fix(G_1\ast G_2)$. The upper bound follows directly from the definition of $G_1\ast G_2$ and the fixing number of a graph.
\end{proof}

The upper bound given in Theorem \ref{fix1} is attained, when both $G_1$ and $G_2$ are complete or null graphs. The next lemma follows from Theorem \ref{nm1}.
\begin{lem}\label{stab}For a rigid graph $G_1$ without dominating vertex and a non-rigid graph $G_2$ without false twins, $Stab(g^1_i,g^2_j)=Stab(g^1_i)\times Stab(g^2_j)$ for each $(g^1_i,g^2_j)\in V(G_1\ast G_2)$.
\end{lem}

In next theorem, we give some conditions on $G_1$ and $G_2$ under which the lower bound given in Theorem \ref{fix1} is attained.
\begin{thm}\label{3}
For any two non-isomorphic graphs $G_1$ and $G_2$,
$fix(G_1\ast G_2)$ $=max\{fix(G_1),$ $fix(G_2)\}$, if one of the following holds: \\
(1) $G_1$ and $G_2$ are rigid graphs.\\
(2) $G_1$ is a rigid graph without dominating vertex and $G_2$ is a non-rigid graph without false twins.\\
(3) $G_1$ and $G_2$ are non-rigid graphs without dominating vertices and false twins.
\end{thm}
\begin{proof}(1) Proof follows from Theorem \ref{f1}.

(2) Let $F_2=\{g^2_{1}, g^2_{2},\ldots, g^2_{l}\}\subseteq V(G_2)$ be a minimum fixing set for $G_2$. For any vertex
$g^1_i\in V(G_1)$, consider set $F(g^1_i)=\{(g^1_i, g^2_{j})|
g^2_{j}\in F_2\}\subseteq V(G_1\ast G_2)$. First, we prove that
$F(g^1_i)$ is a fixing set for $G_1\ast G_2$. As $G_1$ and $G_2$ satisfies the hypothesis of Lemma \ref{stab} yields that $Stab(g^1_i, g^2_j)=Stab(g^1_i)\times Stab(g^2_j)$ for all $(g^1_i, g^2_j)\in V(G_1\ast G_2)$. Moreover, $G_1$ is a rigid graph and $\cap_{g^2_j\in F_2}Stab(g^2_j)= \{id_{G_2}\}$ implies that $\cap_{(g^1_i, g^2_j)\in
F(g^1_i)}Stab(g^1_i, g^2_j)=\{id_{G_1\ast G_2}\}$. Hence,
$F(g^1_i)$ is a fixing set for $G_1\ast G_2$. In order to prove that $F(g^1_i)$ is a minimum fixing set consider any subset $\mathcal{F}$ of $V(G_1\ast G_2)$ such that $|\mathcal{F}|<l$ and define
${\mathcal{F}_1}=\{g^1_i\in V(G_1)| (g^1_i, g^2_j)\in \mathcal{F}$ for
some $g^2_j\in V(G_2)\}$ and $\mathcal{F}_2=\{g^2_j\in V(G_2)|
(g^1_i, g^2_j)\in \mathcal{F}$ for some $g^1_i\in V(G_1)\}$.
Since, $|\mathcal{F}|<l$ so $|\mathcal{F}_2|<l$, moreover, $fix(G_2)=l$ yields that $\cap_{g^2_j\in \mathcal{F}_2}Stab(g^2_j)\neq \{id_{G_2}\}$ and $\cap_{(g^1_i, g^2_j)\in \mathcal{F}}Stab(g^1_i,
g^2_j)\neq\{id_{G_1\ast G_2}\}$ which gives that $\mathcal{F}$ is not a fixing set for $G_1\ast G_2$. Hence, $F(g^1_i)$ is a
minimum fixing set for $G_1\ast G_2$.

(3) Let $F_1=\{g^1_{1}, g^1_{2},\ldots, g^1_{r}\}\subseteq V(G_1)$ and $F_2=\{g^2_{1}, g^2_{2},\ldots, g^2_{l}\}\subseteq V(G_2)$ are minimum fixing sets for $G_1$ and $G_2$, respectively. We define $F=\{(g^1_{1}, g^2_{1}), (g^1_{2},
g^2_{2}),\ldots, (g^1_{r}, g^2_{r})$, $(g^1_{r}, g^2_{r+1}),\ldots,$
$(g^1_{r}, g^2_{l})\}$ $\subseteq V(G_1\ast G_2)$ such that $|F|=l$. Since $F_1$ and $F_2$ are fixing sets so $\cap_{g^1_{i}\in
F_1}Stab(g^1_{i})= \{id_{G_1}\}$ and $\cap_{g^2_{j}\in F_2}
Stab(g^2_{j})= \{id_{G_2}\}$, further, $G_1$ and $G_2$ satisfies the conditions of Corollary \ref{n1} yields that $Stab(g^1_i, g^2_i)=Stab(g^1_i)\times
Stab(g^2_i)$ for each $(g^1_i, g^2_i)\in V(G_1\ast G_2)$. Hence, $\cap_{(g^1_{i}, g^2_{j})\in
F}Stab(g^1_{i}, g^2_{j})=\{id_{G_1\ast G_2}\}$ and $F$ is a fixing set for $G_1\ast G_2$. In order to prove that $F$ is a minimum fixing set, consider any subset $\mathcal{F}\subseteq V(G_1\ast G_2)$
such that $|\mathcal{F}|< l$ and define
$\mathcal{F}_1=\{g^1_i\in V(G_1)| (g^1_i, g^2_j)\in \mathcal{F}$ for some $g^2_j\in V(G_2)\}$ and $\mathcal{F}_2=\{g^2_j\in V(G_2)| (g^1_i, g^2_j)\in \mathcal{F}$
for some $g^1_i\in V(G_1)\}$. Also, $|\mathcal{F}|<l$ gives that
$|{\mathcal{F}_2}|<l$ which yields that $\mathcal{F}_2$ is not a fixing set for $G_2$,
i.e., $\cap_{g^2_j\in \mathcal{F}_2}Stab(g^2_j)\neq \{id_{G_2}\}$. Hence,
$\cap_{(g^1_i, g^2_j)\in \mathcal{F}}Stab(g^1_i,
g^2_j)\neq\{id_{G_1\ast G_2}\}$ which yields that $\mathcal{F}$ is not a fixing set for $G_1\ast G_2$ and we conclude that $F$ is a minimum fixing set for $G_1\ast G_2$.
\end{proof}

In the next theorem, we give the fixing number of $G_1\ast G_2$, when exactly one of $G_1$ or $G_2$ is a rigid graph.
\begin{thm}\label{4}If $G_1$ is a rigid graph with a dominating vertex
and $G_2$ is a non-rigid graph without false twins, then
$fix(G_1\ast G_2)=2fix(G_2)$.
\end{thm}
\begin{proof} Let $g^1_k$ be dominating vertex of $G_1$
and $G^{\prime}= \langle V(G_1)\setminus \{g^1_k\}\rangle$ be induced
subgraph of $G_1$. Observation \ref{o2} yields that $G_1\ast G_2=G^{\prime}\ast G_2+ \langle\mathcal{G}(g^1_k)\rangle$ and using Lemma \ref{2} we have $fix(G_1\ast G_2)=fix(G^{\prime}\ast G_2)+ fix(\langle\mathcal{G}(g^1_k)\rangle$. Further, $\langle\mathcal{G}(g^1_k)\rangle\cong G_2$ and Proposition \ref{1} gives that $G^{\prime}$ is also a rigid graph, moreover, Theorem \ref{3} (2) gives that $fix(G^{\prime}\ast
G_2)=fix(G_2)$. Hence, $fix(G_1\ast G_2)=2fix(G_2)$.
\end{proof}

\begin{exm}Let $G_1$ be a rigid graph with a dominating vertex. If
$G_2\cong P_n$ with $n\geq 4$, then Theorem \ref{4} gives that  $fix(G_1\ast G_2)= 2$, because $fix(P_n)=1$. If
$G_2\cong C_n$ with $n\geq 5$, then Theorem \ref{4} gives that $fix(G_1\ast G_2)=4$, because $fix(C_n)=2$.
\end{exm}

For a vertex $x$ in a graph $G_1$, set $C(x)=\{u\in V(G_1) : N(u)=N(x)\}$ is called the \emph{equivalence class} of false twins of $x$ in $G_1$.
The next theorem, follows from Theorem \ref{twin} and the definition of co-normal product of graphs.
\begin{thm}\label{con}Let $G_1$ be a connected graph and $G_2$ be a graph such that $G^2_1,$ $G^2_2,$ $\ldots,$ $G^2_s$ are distinct equivalence classes of false twins and forms a partition of $V(G_2)$, with $|G^2_j|\geq 2$, for each $j$. If $G_1$ has no false twins, then $fix(G_1\ast G_2)=|V(G_1)|fix(G_2)$.
\end{thm}
\begin{proof}Since $G_1$ has no false twins, therefore, Theorem \ref{twin} yields that $(g^1_i, g^2_j), (g^1_k,g^2_l)\in V(G_1\ast G_2)$ are false twins if and only if $g^1_i=g^1_k$ and $g^2_j, g^2_l\in G^2_r$ for some $1\leq r\leq s$. We define set $\mathcal{G}_{ij}=\{g^1_i\}\times G^2_j$ for each $1\leq j\leq s$ and $g^1_i\in V(G_1)$ which is an equivalence class of false twins of $G_1\ast G_2$ such that $|\mathcal{G}_{ij}|\geq 2$ and $G_1\ast G_2$ has $|V(G_1)|\cdot s$ distinct such classes. Moreover, the collection $\mathcal{G}_{ij}$$'s$ forms a partition of $V(G_1\ast G_2)$ because $G^2_i$$'s$ gives a partition of $G_2$. Hence, $fix(G_1\ast G_2)=\sum_{g^1_i\in V(G_1)}\sum_{j=1}^{s}|\mathcal{G}_{ij}|- |V(G_1)|\cdot s=|V(G_1)||V(G_2)|-|V(G_1)|\cdot s$ implies that $fix(G_1\ast G_2)=|V(G_1)|fix(G_2)$.
\end{proof}

\begin{exm}If $G_1$ is a path graph of order $m$ and
$G_2$ be a null graph of order $n$, then Theorem \ref{con} gives that $fix(G_1\ast G_2)=m(n-1)$, because $fix(N_n)=n-1$.
\end{exm}

In the next two theorems, we give formulae for the fixing number of $G_1\ast G_2$, under certain conditions.
\begin{thm}In a graph $G_1$, let $D$ be the set of all dominating vertices with $|D|\geq 1$ and $G^1_1, G^1_2, \ldots, G^1_r$ with $r\geq 1$ be the distinct equivalence classes of false twins with $|G^1_i|\geq 2$ for each $1\leq i\leq r$. If the collection $D,G^1_1, G^1_2, \ldots, G^1_r$ gives a partition of $V(G_1)$ and $G_2$ has no false twins, then $fix(G_1\ast G_2)= |V(G_2)|(\sum\limits_{i=1}^{r}|G^1_i|-r)+|D|fix(G_2)$.
\end{thm}
\begin{proof}Let $G^\prime=\langle V(G_1)\setminus D\rangle$, then $G^1_1,G^1_2,\ldots,G^1_r$ give a partition of $V(G^\prime)$ and are distinct equivalence classes of false twins in $G^\prime$. Now for each $g^2_j\in V(G_2)$ and for each $G^1_i$, the set $\mathcal{G}_{ij}=G^1_i\times \{g^2_j\}$ is an equivalence class of false twins and there are $|V(G_2)|\cdot r$ such classes in $G^\prime\ast G_2$. Moreover, the collection of $\mathcal{G}_{ij}$$'s$ yields a partition of $V(G^\prime\ast G_2)$ implies that $fix(G^\prime\ast G_2)= |V(G_2)|(\sum_{i=1}^{r}|G^1_i|-r)$. Now, for each $g^1_i\in D$, Lemma \ref{2} gives that the vertices of $\mathcal{G}(g^1_i)$ are not fixed by any $x\in V(G_1\ast G_2)\setminus \mathcal{G}(g^1_i)$ and $\mathcal{G}(g^1_i)\cong G_2$ gives that $fix(\langle D\times V(G_2)\rangle)=|D|fix(G_2)$. Observation \ref{o2} extended to a set of dominating vertices gives that $G_1\ast G_2=\langle(V(G_1)\setminus D)\times V(G_2)\rangle+\langle D\times V(G_2)\rangle$$=$$G^\prime\ast G_2+\langle D\times V(G_2)\rangle$. Hence, $fix(G_1\ast G_2)=|V(G_2)|(\sum\limits_{i=1}^{r}|G^1_i|-r)+|D|fix(G_2)$.
\end{proof}
\begin{thm}Let $G_1$ and $G_2$ be two graphs such that $G_1$ has $r$
distinct equivalence classes $G^1_1, G^1_2,\ldots, G^1_r$ of false
twins and $G_2$ has $s$ distinct equivalence classes $G^2_1, G^2_2,\ldots, G^2_s$ of false twins. If $\cup_{i=1}^{r}G^1_i=V(G)$, $\cup_{j=1}^{s}G^2_j=V(G_2)$ and $|G^1_i|, |G^2_j|\geq 2$ for each
$i$ and $j$, then $fix(G_1\ast G_2)=|V(G_1)||V(G_2)|-rs$.
\end{thm}
\begin{proof}Theorem \ref{twin} gives that for each $i$ and $j$, the set $G^1_i\times G^2_j$ is an
equivalence class of false twins in $G_1\ast G_2$, moreover, the
collection $G^1_i\times G^2_j$, $1\leq i\leq r$, $1\leq j\leq s$
gives a partition of $V(G_1\ast G_2)$ such that $|G^1_i\times G^2_j|\geq
4$. Hence, $fix(G_1\ast G_2)=|V(G_1)||V(G_2)|-rs$.
\end{proof}
The next corollary deals with the fixing number of the co-normal product of two multipartite graphs.
\begin{cor}For $G_1\cong K_{k_1, k_2,\ldots, k_r}$ and $G_2\cong K_{l_1,
l_2, \ldots, l_s}$, $fix(G_1\ast G_2)$ $= |V(G_1)|$ $|V(G_2)|-rs$.
\end{cor}
In \cite{FD}, Erwin and Harary gave the following observation.
\begin{obs}\cite{FD}\label{o4}
	For a graph $G$, $fix(G)=1$ if and only if $G$ has an orbit of cardinality $|Aut(G)|$.
\end{obs}

\begin{figure}[h]
        \centerline
        {\includegraphics[width=14cm]{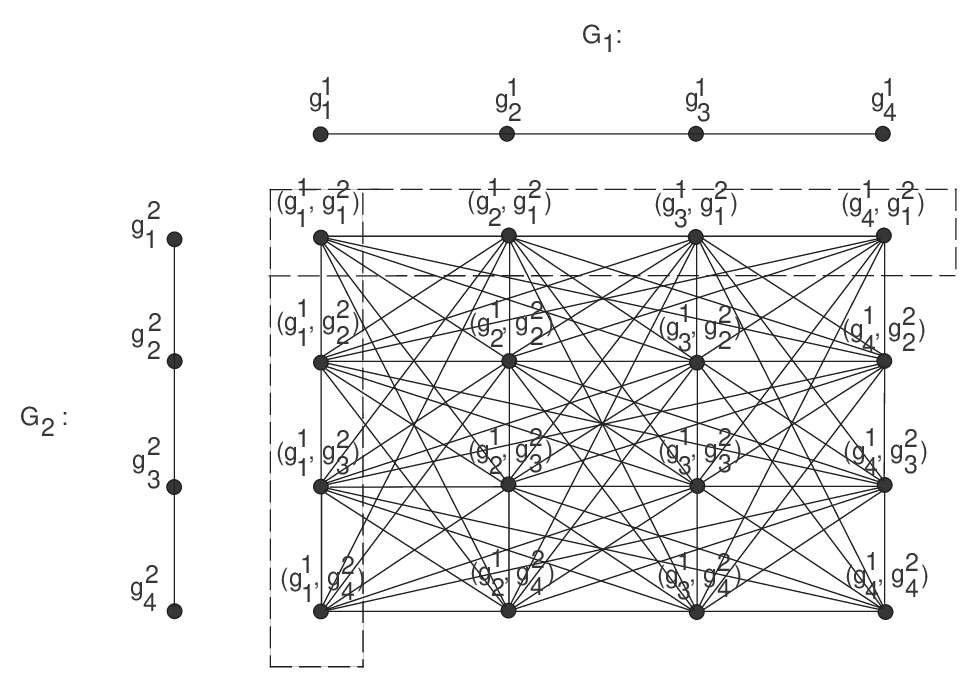}}
        \label{p}
        \caption{The co-normal product graph of $G_1=P_4$ and $G_2=P_4$.}\label{p}
\end{figure}
Note that the vertices $(g^1_1,g^2_1)$, $(g^1_1,g^2_4)$, $(g^1_4,g^2_1)$ and $(g^1_4,g^2_4)$ for the co-normal product graph $P_4\ast P_4$ shown in Figure 2, are similar. Also, set $\{(g^1_1,g^2_2)\}$ is a fixing set implies that $fix(P_4\ast P_4)=1$. In the next theorem, we give the fixing number of the co-normal product of two path graphs.

\begin{thm}Let $P_n$ be a path graph with $n$ vertices, then
for any two integers $s, t\geq 2$,
$$ fix(P_s\ast P_t) = \left\{
            \begin{array}{ll}
            1           &if\,\,\,\mbox{s, t$\geq$ 4},\\
            2            &if\,\,\,\mbox{s$=$2, t$\geq$ 4},\\
            3            &if\,\,\,\mbox{s$=$t$=$2 or s$=$2, t$=$
            3},\\
            5            &if\,\,\,\mbox{s$=$t$=$3},\\
            t+1           &if\,\,\,\mbox{s$=$3, t$\geq$ 4}.
            \end{array}
             \right.
$$
\end{thm}
\begin{proof}For $s, t\geq 4$, we discuss two cases:\\
\\
Case 1) Suppose $s\neq t$,
then $Aut(P_s\ast P_t)\cong S_2\oplus S_2$ and the four nodes of degree $t+s-1$ are similar so by Observation \ref{o4}, $fix(P_t\ast P_s)$=1.\\
\\
Case 2) Suppose $s=t$, then $Aut(P_s\ast P_t)$ is isomorphic to the dihedral group of order 8 and the eight nodes of degree $3s-2$ are similar which are adjacent to the four nodes of degree $2s-1$ (the corner nodes), so by Observation \ref{o4},
$fix(P_s\ast P_t)=1$.\\
\\
Suppose $s=2$ and $t\geq 4$, then $P_s$ has two dominating vertices and Lemma \ref{2} gives that $fix(P_s\ast P_t)=2$.\\
\\
Suppose $s=t=2$, then the order of $P_s\ast P_t$ is 4 and it is a complete graph. Hence, $fix(P_s\ast P_t)=3$.\\
\\
Suppose $s=2$ and $t=3$, then $P_s\ast
P_t$ has two equivalence classes of false twins with cardinality 2
and one class of true twins because $P_t$ also has a dominating
vertex. Hence, $fix(P_s\ast P_t)=3$.\\
\\
Suppose $s=t=3$, then $P_s\ast P_t$ has an equivalence class of false
twins with cardinality 4, two classes of false twins with
cardinality 2 and one dominating vertex. Hence, $fix(P_s\ast P_t)=5$.\\
\\
Suppose $s=3$ and $t\geq 4$, then $P_s\ast P_t$ has $t$
equivalence of cardinality 2. Since $P_s$ has a dominating vertex so
Lemma \ref{2} gives that $fix(P_s\ast P_t)=t+1$.
\end{proof}

In the next theorem, we give formulae for the fixing number of the co-normal product of two graphs, when $G_1$ is a graph without false twins and $G_2$ is a star graph.
\begin{thm}If $G_1$ is a graph without false twins
and $G_2\cong S_{n+1}$, then $fix(G_1\ast G_2)=|V(G_1)|fix(G_2)+fix(G_1)$.
\end{thm}
\begin{proof}Let $V(G_1)=\{g^1_1, g^1_2,\ldots, g^1_{m}\}$ and $V(G_2)=\{g^2_1,
g^2_2,\ldots, g^2_n, g^2_{n+1}\}$ and $g^2_{n+1}$ be the
dominating vertex in $G_2$. Let $F$ be a minimum fixing set for $G_1\ast G_2$. Theorem \ref{twin} gives that for each $g^1_i\in V(G_1)$, the set $\mathcal{G}_i=\{g^1_i\}\times (V(G_2)\setminus \{g^2_{n+1}\})$ is an equivalence class of false twins in $G_1\ast G_2$ which gives that $|F\cap \mathcal{G}_i|\geq n-1$ for each $1\leq i\leq m$. Lemma \ref{2} gives that the vertices of the set $\mathcal{G}(g^2_{n+1})$ are not fixed by any vertex in $V(G_1\ast G_2)\setminus \mathcal{G}(g^2_{n+1})$ and $\mathcal{G}(g^2_{n+1})\cong G_1$ yields that $|F\cap \mathcal{G}(g^2_{n+1})|\geq fix(G_1)$. Also, $\mathcal{G}(g^2_{n+1})\cap \mathcal{G}_i=\emptyset$ for each $i$ implies that $fix(G_1\ast G_2)\geq |V(G_1)|fix(G_2)+fix(G_1)$. Let $F_1$ be a minimum fixing set for $G_1$ and $F_2$ be a minimum fixing set for $G_2$. Consider the set $F=\cup_{i=1}^{m}(\{g^1_i\}\times F_2)\cup (F_1\times \{g^2_{n+1}\})$ such that $Stab(F)=Stab(F_1)\times Stab(F_2)=\{id_{G_1\ast G_2}\}$ yields that $F$ is a fixing set for $G_1\ast G_2$ and $fix(G_1\ast G_2)\leq |V(G_1)|fix(G_2)+fix(G_1)$. Hence, $fix(G_1\ast G_2)= |V(G_1)|fix(G_2)+fix(G_1)$.
\end{proof}

\begin{exm}Let $G_1\cong S_{m+1}$ and $G_2\cong S_{n+1}$ whereas $V(G_1)=\{g^1_1$$, g^1_2,$ $\ldots,$ $g^1_{m},$$ g^1_{m+1}\}$ and
$V(G_2)$ $=\{g^2_1,$ $ g^2_2, $ $\ldots,$$ g^2_{n}$,$g^2_{n+1}\}$ such that
$g^1_{m+1}$ is dominating vertex in $G_1$ and $g^2_{n+1}$ is
dominating in $G_2$. The equivalence classes of false twins in $G_1$
are $G^1_1=\{g_1, g_2, \ldots, g_{m}\}$,
$G^1_2=\{g_{m+1}\}$ and $G^2_1=\{g^2_1, g^2_2, \ldots,
g^2_{n}\}$, $G^2_2=\{g^2_{n+1}\}$ are equivalence classes of
false twins in $G_2$. The classes $G^1_1\times G^2_1$, $G^1_1\times
G^2_2$, $G^1_2\times G^2_1$ and $G^1_2\times G^2_2$ are equivalence
classes of false twins such that $|G^1_2\times G^2_2|=1$. Also,
$G^1_1\times G^2_2= \mathcal{G}(g^2_{n+1})\setminus
\{(g^1_{m+1},g^2_{n+1})\}$, $G^1_2\times G^2_1=
\mathcal{G}(g^1_{m+1})\setminus \{(g^1_{m+1},g^2_{n+1})\}$ and
$(g^1_{m+1},g^2_{n+1})$ is the dominating vertex in $G_1\ast G_2$. Using Lemma \ref{2}, we have $fix(G_1\ast G_2)=mn+fix(G_1)+fix(G_2)-1$.
\end{exm}

%\section*{References}

\bibliography{mybibfile}

\end{document}